\documentclass[reqno,11pt]{amsproc}
\usepackage{latexsym,amssymb,amsfonts,amsmath,amsthm}
\usepackage{mathrsfs}
\usepackage[usenames]{color}
\usepackage{eucal}
\usepackage{mathptmx} 
\usepackage{comment}
\usepackage[colorlinks,citecolor=blue,urlcolor=blue]{hyperref}

\usepackage{tikz}
\usetikzlibrary{decorations.markings}
\usepackage{pgfplots}
\pgfplotsset{compat=1.11}

\newtheorem{theorem}{Theorem}
\newtheorem{cor}[theorem]{Corollary}
\newtheorem{lemma}[theorem]{Lemma}
\theoremstyle{definition}
\newtheorem{example}[theorem]{Example}
\theoremstyle{remark}
\newtheorem{remark}[theorem]{Remark}

\newcommand{\Real}{\mathbb R}
\newcommand{\eps}{\varepsilon}

\newcommand{\F}{\mathcal{F}}

\renewcommand{\P}{\mathsf{P}}

\newcommand{\E}{\mathsf{E}}

\newcommand{\var}{\mathsf{Var}}

\begin{document}

\title[]{An approximation of populations on a habitat with large carrying capacity}%

\author{N.Bauman}%
\address{Department of Statistics,
The Hebrew University,
Mount Scopus, Jerusalem 91905,
Israel}
\email{Naor.Bauman@mail.huji.ac.il}

\author{P.Chigansky}%
\address{Department of Statistics,
The Hebrew University,
Mount Scopus, Jerusalem 91905,
Israel}
\email{Pavel.Chigansky@mail.huji.ac.il}

\author{F.Klebaner}%
\address{School of Mathematics, 
Monash University, Clayton, Vic. 3058, Australia
}
\email{fima.klebaner@monash.edu} 

\subjclass{92D25, 60J80}%
\keywords{population dynamics, branching processes, limit theorems, approximation}

\date{\today}%

\begin{abstract}
We consider stochastic dynamics of a population which starts from a small colony 
on a habitat with large but limited carrying capacity. A common heuristics suggests that 
such population grows initially as a Galton-Watson branching process and then its size follows an almost 
deterministic path until reaching its maximum, sustainable by the habitat. 
In this paper we put forward an alternative and, in fact, more accurate approximation 
which suggests that the population size behaves as a special nonlinear transformation of the Galton-Watson process from the very beginning.
\end{abstract}
\maketitle
 
\section{Introduction}

\subsection{The model}

A large population often starts from a few individuals who colonize a new habitat.
Initially, in abundance of resources and lack of competition it grows rapidly 
until reaching the carrying capacity. Then the population  
 fluctuates around
the carrying capacity for a very long period of time, until, by chance, it eventually dies
out, see, e.g., \cite{general}, \cite{HJK}.

This cycle is captured by a stochastic model of density dependent branching process $Z=(Z_n, n\in \mathbb Z_+)$ 
generated by the recursion 
\begin{equation}\label{Z}
Z_n = \sum_{j=1}^{Z_{n-1}} \xi_{n,j}, \quad n\in \mathbb N,
\end{equation}
started at an initial colony size $Z_0$. The random variables $\xi_{n,j}$ take integer values and, 
for each $n\in \mathbb N$, are conditionally i.i.d. given all previous generations
$$
\F_{n-1} =\sigma\{\xi_{m,j}: m<n, j\in \mathbb N\}.
$$

The object of our study is the {\em density} process of the population $\overline Z_{n}:= Z_{n}/K$ 
relative to the carrying capacity parameter $K>0$.
The common distribution of the random variables $\xi_{n,j}$ is assumed to depend on the density $\overline Z_{n-1}$:
\begin{equation}\label{pell}
\P(\xi_{n,1}=\ell|\F_{n-1}) = p_\ell(\overline Z_{n-1}), \quad \ell\in \mathbb Z_+,
\end{equation}
and is determined by the   functions $p_\ell:\Real_+\mapsto[0,1]$.

Both processes $Z$ and  $\overline Z$ are indexed by $K$, but this dependence is suppressed in the notation. 
The mean and the variance of offspring distribution when the density process has value $x$ are denoted by 
\begin{equation}\label{ms}
m(x) = \sum_{k=0}^\infty k p_k(x)\quad\text{and}\quad \sigma^2(x) = \sum_{k=0}^\infty (k-m(x))^2 p_k(x), \quad x\in \Real_+,
\end{equation}
assumed to exist. Consequently,
$$
\E(\xi_{n,1}|\F_{n-1}) = m(\overline Z_{n-1}) \quad \text{and}\quad \var(\xi_{n,1}|\F_{n-1})=\sigma^2(\overline Z_{n-1}).
$$

If the offspring mean function satisfies 
\begin{equation}\label{mx}
m(x) \begin{cases}
> 1, & x<1\\
=1, & x=1\\
<1, & x>1
\end{cases}
\end{equation}
the process $Z$ has a supercritical reproduction below the capacity  $K$, critical reproduction at $K$ and a subcritical reproduction above $K$.
Thus a typical trajectory of $Z$ grows rapidly until it reaches the vicinity of $K$.
It then stays there fluctuating for a very long period of time and gets extinct eventually if $p_0(x)>0$ for all $x\in \Real_+$. 
Thus the lifespan of such population roughly divides between the emergence stage, at which the population establishes itself,
the quasi-stationary stage around the carrying capacity and the decline stage which ends up with 
inevitable extinction. 

\begin{remark}
While \eqref{mx} is typical for populations with quasi stable equilibrium 
at the capacity, it is not needed in the proofs and will not be assumed in what follows. 
\end{remark}

\subsection{Large initial colony}
A more quantitative picture can be obtained by considering the dynamics  for the density process derived from \eqref{Z} by setting  $f(x):=xm(x)$, dividing  by $K$ and rearranging:
\begin{equation}\label{Zbar}
\overline Z_n = f(\overline Z_{n-1}) + \frac 1 K \sum_{j=1}^{Z_{n-1}} (\xi_{n,j}-m(\overline Z_{n-1})).
\end{equation}
The second term on the right has zero mean and conditional variance  
$$
\var\Big(\frac 1 K \sum_{j=1}^{Z_{n-1}} (\xi_{n,j}-m(\overline Z_{n-1}))\Big|\F_{n-1}\Big)=K^{-1}\overline Z_{n-1} \sigma^2(\overline Z_{n-1}).
$$
Consequently \eqref{Zbar} can be viewed as a deterministic dynamical system perturbed by small noise
of order\footnote
{
The usual notations for probabilistic orders is used throughout. In particular, for a sequence of random variables $\zeta(K)$  
and a sequence of numbers $\alpha(K)\searrow 0$ as $K\to\infty$, the notation $\zeta(K)=O_\P(\alpha(K))$ means that the sequence $\alpha(K)^{-1}\zeta(K)$ is bounded in probability.
}
$O_\P(K^{-1/2})$. If the initial colony size is relatively large, i.e., proportional to the carrying capacity: 
$$
\overline Z_0=Z_0/K \xrightarrow[K\to\infty]{} x_0>0,
$$
then $\overline Z_n \xrightarrow[K\to\infty]{\P} x_n$
where $x_n$ follows the unperturbed deterministic dynamics 
\begin{equation}\label{xn}
x_n = f(x_{n-1}), \quad n\in \mathbb N, 
\end{equation}
started at $x_0$.
If  \eqref{mx} is assumed, $x=1$ is the stable fixed point of $f$ and if, in addition, $f$ is an increasing
function,  then the sequence $x_n$ increases to 1 with $n$ when $x_0<1$.
This limit also implies  that the probability of early extinction tends to zero as $K\to\infty$. 

Moreover, the stochastic fluctuations about the deterministic limit converge to a Gaussian process $V=(V_n, n\in \mathbb{Z}_+)$ in distribution:
$$
\sqrt K (\overline Z_n-x_n)\xrightarrow[K\to\infty]{d}V_n
$$
where $V_n$ satisfies the recursion,  \cite{KN94},
$$
V_n = f'(x_{n-1}) V_{n-1} + \sqrt{x_{n-1} \sigma^2(x_{n-1})}W_n,\quad n\in \mathbb N,
$$
with $N(0,1)$ i.i.d. random variables $W_n$'s.

Roughly speaking, this implies that when $K$ is large, $Z_n$ grows towards the 
capacity $K$ along the deterministic path $K x_n$ and its fluctuations are of order $O_\P(K^{1/2})$:
\begin{equation}\label{ZnK}
Z_n =  x_n K + V_n  K^{1/2} + o_\P(K^{1/2}), \quad n\in \mathbb N.
\end{equation}
If $p_0(x)>0$ for all $x\in \Real_+$ and  \eqref{mx} is imposed, zero is an absorbing state and hence the population gets extinct eventually. 
Large deviations analysis, see for example \cite{LDP}, \cite{LDP1}, and \cite{Jung}, \cite{HJ}, shows that the mean of the 
time to extinction $\tau_e =\inf\{n\ge 0: Z_n=0\}$ grows exponentially with $K$. 
In this paper we are concerned with the establishment stage of the population, which occurs well before the ultimate extinction, on the time scale of $\log K$.

\subsection{Small initial colony} 
When $Z_0$ is a fixed integer, say $Z_0=1$, then $Z_0/K\to x_0=0$ and, since $f(0)=0$, the solution to \eqref{xn}
is $x_n=0$ for all $n\in \mathbb N$. In this case the approximation \eqref{ZnK} ceases to provide useful information. 
An alternative way to describe the stochastic dynamics in this setting was suggested recently in \cite{BChK}, \cite{ChJK18, ChJK19}.
It is based on the long known heuristics \cite{heur1},  \cite{heur2},  \cite{heur3}, according to which such a population  behaves initially as the Galton-Watson branching 
process and, if it manages to avoid extinction at this early stage, it continues to grow towards the carrying capacity 
following an almost deterministic curve.

This heuristics is made precise in  \cite{ChJK19}  as follows.
We  couple $Z$ to a supercritical Galton-Watson branching process $Y=(Y_n, n\in \mathbb Z_+)$ 
started at $Y_0=Z_0=1$,
\begin{equation}\label{Y}
Y_n = \sum_{j=1}^{Y_{n-1}} \eta_{n,j}
\end{equation}
with the offspring distribution identical to that of $Z$ at zero density size
$$
\P(\eta_{1,1}=\ell) = p_\ell(0), \quad \ell\in \mathbb Z_+.
$$
This coupling is defined under assumption \eqref{a1} below in Section \ref{sec:3.2}.

Denote by $\rho:=m(0)>1$, define
\footnote{$[x]$ and $\{x\}=x-[x]$ denote the integer and fractional part of $x\in \Real_+$. }
$n_c := n_c(K)=[\log_\rho K^c]$ for some $c\in (\frac 1 2,1)$ and let $\overline Y_n :=Y_n/K$ be the density 
of $Y$. Then $\overline Z_n=Z_n/K$ is approximated in \cite{ChJK19} by  
$$
\begin{cases}
\overline Y_n,& n\le n_c, \\
f^{n-n_c}(\overline Y_{n_c}), & n>n_c,
\end{cases}
$$
where $f^k$ stands for the $k$-th iterate of $f$. 
As is well known \cite{ANbook}
$$
\rho^{-n} Y_n \xrightarrow[n\to\infty]{\P-\text{a.s.}} W,
$$ 
where $W$ is an  a.s. finite random variable.
Moreover, under certain technical conditions on $f$, the limit 
\begin{equation}\label{limf}
H(x):=\lim_{n\to\infty}f^n(x/\rho^n), \quad x\in\Real_+
\end{equation}
can be shown to exist and  define a continuous function. 

\medskip

\begin{theorem}[\cite{ChJK19}]\label{thm:1}
Let $n_1:=n_1(K)=[\log_\rho K]$ then 
\begin{equation}\label{lim}
\overline Z_{n_1} - H\Big(W\rho^{-\{\log_\rho K\}}\Big) \xrightarrow[K\to\infty]{\P}0.
\end{equation}
\end{theorem}

In particular, this result implies that when $K$ is a large integer power of $\rho$
the distribution of $\overline Z_{n_1}$ is close to that of $H(W)$. Moreover,  
$$
\overline Z_{n_1+n}\xrightarrow[K\to\infty]{\P} x_n, \quad n\in \mathbb N,
$$ 
where $x_n$ solves \eqref{xn} started from the {\em random} initial condition $H(W)$. 
This approximation also captures the early extinction event since $H(0)=0$ and 
$\P(W=0)=\P(\lim_n Y_n=0)$, the extinction probability of the Galton-Watson process $Y$.

\subsection{This paper's contribution}

In this paper we address the question of the  rate of convergence in \eqref{lim}.
Note that if the probabilities in \eqref{pell} are constant with respect to $x$ 
then $f(x)=\rho x$,   consequently $H(x)=x$, and the processes $Z$ and $Y$ coincide. 
In this case
\begin{equation}\label{H}
\sqrt K (\overline Y_{n_1} -W\rho^{-\{\log_\rho K\}}) = 
 \rho^{-\frac 1 2\{\log_\rho K\}}   \sqrt{\rho^{n_1}} (\rho^{-n_1}Y_{n_1}  -W) = O_\P(1)
\end{equation}
where the order of convergence is implied by the CLT for the Galton-Watson process \cite{H70} by which 
$\sqrt{\rho^{n}} (\rho^{-n}Y_{n}  -W)$ converges in distribution to a mixed normal law as $n\to\infty$. 
Thus it can be expected that at best the sequence in \eqref{lim} is 
of order $O_{\P}(K^{-1/2})$ as $K\to\infty$. 
However, the best rate of convergence in the approximation in Theorem \ref{thm:1} described above, is achieved with $c=\frac 58$ and it is only $O_\P(K^{-1/8}\log K)$.
This can be seen  from a close  examination of the proof in \cite{ChJK19}.

The goal of this paper is to put forward a different approximation 
with much faster rate of convergence  of order $O_\P(K^{-1/2}\log K)$. This is still slower than the rate achievable in the 
density independent case, but   only by a logarithmic factor. 
The new proof highlights a better understanding of population dynamics at the emergence stage, 
which shows that, in fact, a sharper approximation is given by  the Galton-Watson process transformed by a nonlinear function $H$ arising in deterministic dynamics \eqref{limf}. 

It is not clear at the moment whether the $\log K$ factor is avoidable and whether a central limit type theorem 
holds. These questions are left for further research.

\section{The main result}

We will make the following assumptions. 

\medskip

\begin{enumerate}
\addtolength{\itemsep}{0.7\baselineskip}
\renewcommand{\theenumi}{a\arabic{enumi}}

\item\label{a1} The offspring distribution 
$
F_x(t) = \sum_{\ell\le t} p_\ell (x)
$
is stochastically decreasing with respect to the population density: for any $y\ge x$,
$$
F_y(t)\ge  F_x(t), \quad \forall t\in \Real_+.
$$

\item\label{a2} 
The second moment of the offspring distribution, cf. \eqref{ms},
$$
m_2(x) = \sigma^2(x)+m(x)^2
$$
is $L$-Lipschitz for some $L>0$.

\item\label{a3} The function $f(x)=xm(x)$ has two continuous bounded derivatives and 
\footnote{$\|f\|_\infty =\sup_x |f(x)|$}
$$
\|f'\|_\infty = f'(0)=\rho. 
$$
\end{enumerate}

\medskip

\begin{remark}
Assumption \eqref{a1} means that the reproduction drops with population density. 
In particular, it implies that $x\mapsto m(x)$ is a decreasing function and hence,  
$$
f'(x)=m(x)+xm'(x) \le m(x)\le \rho,\quad \forall x\in \Real_+,
$$
which is only slightly weaker than \eqref{a3}. The assumption \eqref{a2} is technical. 
\end{remark}

\medskip

\begin{remark}
The distribution of the process $\overline Z$ does not depend on the values of $\{p_\ell(0), \ \ell\in \mathbb Z_+\}$ for any $K$, while the distribution of $W$ and, therefore, of $H(W)$ does. This is not a contradiction since our assumptions imply   
continuity of $x\mapsto p_\ell(x)$ at $x=0$ for all $\ell\in \mathbb Z_+$. 
Indeed, 
$
m(x) = \int_0^\infty (1-F_x(t))dt
$
and therefore 
$$
\int_0^\infty (F_x(t)-F_0(t))dt = m(0)-m(x)\xrightarrow[x\to 0]{}0
$$
where the convergence holds since $m(x)$ is differentiable and a fortiori continuous at $x=0$.
By the stochastic order assumption \eqref{a1}, $F_x(t)-F_0(t)\ge 0$ for any $t\ge 0$.
Since both $F_x$ and $F_0$ are discrete with jumps at integers, for any $s\ge 0$, 
$$
F_x(s)-F_0(s) = \int_{[s]}^{[s]+1} (F_x(t)-F_0(t))dt\le  \int_0^\infty (F_x(t)-F_0(t))dt \xrightarrow[x\to 0]{} 0.
$$
This in turn implies that $p_\ell(x)\to p_\ell(0)$ as $x\to 0$ for all $\ell$.  
\end{remark}

\medskip

\begin{theorem}\label{thm:main}
Under assumptions  \eqref{a1}-\eqref{a3}
$$
\overline Z_{n_1} - H\Big(W\rho^{-\{\log_\rho K\}}\Big) = O_\P\big(K^{-1/2}\log K\big), \quad \text{as\ }\ K\to \infty.
$$
\end{theorem}

\begin{example}
The binary splitting model from \cite{ChJK18} satisfies the above assumptions. 
Another example is Geometric offspring distribution 
$$
p_\ell(x) = q(x)^\ell (1-q(x)), \quad \ell \in \mathbb Z_+
$$
where $q:\Real_+\mapsto [0,1]$ is a decreasing function. This distribution satisfies the stochastic order condition \eqref{a1}. 
The normalization $m(0)=\rho$ and $m(1)=1$ implies that $q(0)=\rho/(1+\rho)$ and $q(1)=1/2$.
A direct calculation shows that, e.g., 
$$
q(x) = \frac{\rho}{1+\rho}\exp\left(-x \log  \frac {2\rho}{1+\rho} \right), \quad x\ge 0
$$
satisfies both \eqref{a2} and \eqref{a3}.
\end{example}
 
\begin{example} Stochastic Ricker model \cite{Hog} is given by a density dependent branching process with the offspring distribution 
$$
p_\ell(x) = q_\ell e^{-\gamma x},
$$
where $\gamma>0$ is a constant, $q_\ell$, $\ell \ge 1$ is a given probability distribution, and no offspring are produced with probability 
$1-e^{-\gamma x}$.
This model satisfies the stochastic ordering assumption \eqref{a1}.
The mean value of the distribution $q_\ell$ is denoted by $e^r$, to emphasize the relation to the deterministic Ricker model. With such notation, 
$$
m(x)=e^{r-\gamma x},\;\;f(x)=xe^{r-\gamma x}.
$$
Under normalization $m(0)=\rho$ and $m(1)=1$ this becomes 
$$
m(x)=\rho^{1-x},\;\;f(x)=x \rho^{1-x}.
$$
A direct calculation verifies the assumptions \eqref{a2} and \eqref{a3}.
 \end{example}

\section{Proof of Theorem \ref{thm:main}}

We will construct the process $Z$ defined in \eqref{Z} and the Galton-Watson process $Y$ from \eqref{Y}
on a suitable probability space so that $Y_n\ge Z_n$ for all $n\in \mathbb N$ and the trajectories of these processes remain
sufficiently close at least for relatively small $n$'s  (Section \ref{sec:3.2}). We will then show that $H$ is twice continuously differentiable (Section \ref{sec-H}) and use Taylor's approximation 
to argue (Section \ref{sec-T}) that   
$$
\overline Z_{n_1} - H(\overline Y_{n_1}) = O_\P(K^{-1/2} \log K), \quad \text{as\ } K\to\infty.
$$
This convergence combined with \eqref{H} implies the result. Below we will write $C$ for a generic constant whose value may change 
from line to line.

\subsection{Properties of $\mathbf H$}\label{sec-H}
In this section we establish existence of the limit \eqref{limf} under the standing assumptions
and verify its smoothness. The proof of existence relies on a result on functional iteration, shown in \cite{BChJK20}.

\medskip

\begin{lemma}\label{lem:bnd}\cite[Lemma 1]{BChJK20} 
Let $x_{m,n}$ be the sequence generated by the recursion 
$$
x_{m,n} = \rho x_{m-1,n}(1+C x_{m-1,n}), \quad m=1,...,n
$$ 
subject to the initial condition $x_{0,n}=x/\rho^n>0$, where $\rho>1$ and $C\ge 0$ are constants. 
There exists a locally bounded function $\psi:\Real_+\mapsto \Real_+$ such that for any $n\in \mathbb N$
\begin{equation}\label{xbnd}
x_{m,n} \le \psi(x)\rho^{m-n}, \quad m=1,...,n.
\end{equation}
\end{lemma}

\medskip

Throughout we will use the notation $H_n(x) :=f^n(x/\rho^n)$.

\medskip

\begin{lemma}\label{lem:1}
Under assumption \eqref{a3} there exists a continuous function $H:\Real_+\mapsto \Real_+$ 
and a locally bounded function $g:\Real_+\mapsto\Real_+$ such that 
$$
\big|H_n(x) - H(x)\big|\le g(x) \rho^{-n},\quad n\in \mathbb  N.
$$
\end{lemma}

\begin{proof}
By assumption \eqref{a3}
\begin{equation}\label{fr}
f(x) = \rho x + \int_0^x \int_0^t f''(s)ds dt
\end{equation}
and hence for any $x,y\in \Real_+$  
\begin{equation}\label{fbnd}
|f(y)-f(x)| \le   \rho|y-x| +\frac 1 2 \|f''\|_\infty |y^2-x^2| \le 
\rho \big(1 + C|y|\vee |x|\big)|y-x|
\end{equation}
with $C=\|f''\|_\infty/\rho$. 
Thus the sequence $x_{m,n}:= f^m(x/\rho^n)$ satisfies 
$$
x_{m,n} = f (x_{m-1,n}) \le 
\rho\big(1 + C   x_{m-1,n}  \big)  x_{m-1,n}
$$
and $x_{0,n}=x/\rho^n$. By Lemma \ref{lem:bnd} there exists a locally bounded function $\psi$ such that 
for any $n\in \mathbb N$ 
\begin{equation}\label{fmbnd}
\big|f^m(x/\rho^n)\big| \le \psi(x)\rho^{m-n}, \quad m=1,...,n.
\end{equation}
The bound \eqref{fbnd} also implies
\begin{equation}\label{ff}
\begin{aligned}
\big|f^{m+1}(x/\rho^{n+1}) - f^m (x/\rho^n) \big|=  &
\big|f\circ f^{m}(x/\rho^{n+1}) - f\circ f^{m-1} (x/\rho^n) \big|
\le \\
&
\rho \big(1+ C F_{m,n}\big) 
\big|f^m(x/\rho^{n+1}) - f^{m-1} (x/\rho^n) \big|
\end{aligned}
\end{equation}
where, in view of \eqref{fmbnd}, 
$$
F_{m,n} := f^{m}(x/\rho^{n+1})\vee f^{m-1} (x/\rho^n)  \le   \psi(x) \rho^{m-1-n}.
$$
Since $f$ has bounded second derivative and $f'(0)=\rho$, cf. \eqref{fr},
$$
|f(x/\rho^{n+1}) -  x/\rho^n| \le \frac 1 2 \|f''\|_\infty (x/\rho)^2 \rho^{-2n}.
$$
Plugging this bound into \eqref{ff} and iterating $n$ times we obtain
\begin{align*}
& \big|f^{n+1}(x/\rho^{n+1}) - f^n (x/\rho^n) \big| \le \big|f(x/\rho^{n+1}) -  x/\rho^n\big|\, \rho^n \prod_{m=1}^n \big(1+ C F_{m,n}\big)\le \\
&
\frac 1 2 \|f''\|_\infty (x/\rho)^2 \rho^{-2n}\rho^n \prod_{m=1}^n   \big(1+ C \psi(x) \rho^{m-1-n}\big) 
\le \widetilde g(x) \rho^{-n}
\end{align*}
where we defined 
$$
\widetilde g(x) := 
\frac 1 2 \|f''\|_\infty (x/\rho)^2    \prod_{k=1}^\infty   \big(1+ C \psi(x) \rho^{-k}\big).
$$
Thus the limit $H(x)=\lim_{n\to\infty} f^n(x/\rho^n)$ exists and satisfies the claimed bound  with $g(x)=\widetilde g(x)/(1-\rho^{-1})$.
Continuity of $H$ follows since $H_n$ are continuous for each $n$ and the convergence is uniform on compacts. 
\end{proof}

\begin{cor}\label{cor}
  $f$  is topologically semiconjugate to its linearization at the origin:
$$
H(x)=f\circ H(x/\rho), \quad \forall x\in \Real_+.
$$
\end{cor}

\begin{proof}
Since $f$ is continuous 
$$
H(x)=\lim_{n\to\infty}f^{n+1}(x/\rho^{n+1}) = \lim_{n\to\infty} f\circ f^n((x/\rho)\rho^{-n}) = f\circ H(x/\rho).
$$
\end{proof}

The next lemma shows that $H$ is strictly increasing in a vicinity of the origin and is therefore a local conjugacy. 

\medskip

\begin{lemma}\label{lem:8}
There exists an $a>0$ such that $H$ is strictly increasing on $[0,a]$ and 
\begin{equation}\label{conj}
f(x) = H (\rho H^{-1}(x)), \quad x\in [0,H(a)].
\end{equation}
\end{lemma}

\begin{proof}
Let $c:= \|f''\|_\infty$ and $r := \rho /c$ then 
$$
f'(x) \ge \rho - c x>0, \quad \forall x\in [0,r).
$$
Since $f$ is $\rho$-Lipschitz and $f(0)=0$, for any $j=1,...,n$ and $x\in [0,r)$,
$$
f^{n-j}(x/\rho^n)\le x/\rho^j\in [0,r)
$$
and hence for all $x\in [0,r)$
\begin{align*}
H_n'(x) = & \prod_{j=1}^n \frac 1 \rho f'(f^{n-j}(x/\rho^n)) \ge \prod_{j=1}^n \big(1-\frac c \rho  f^{n-j}(x/\rho^n) \big) \ge \\
&
\prod_{j=1}^n \big(1-\frac c \rho  x\rho^{-j} \big)  \ge 1-\frac c \rho  x \sum_{j=1}^n \rho^{-j} \ge 1-\frac c {\rho-1}  x,
\end{align*}
where we used the Bernoulli inequality. 
Thus we can choose a number $a\in (0,r)$ such that $H_n'(x)\ge 1/2$ for all $x\in [0,a]$. It then follows that for any $y>x$ in 
the interval $[0,a]$ 
$$
H_n(y)-H_n(x) = \int_x^y H_n'(t)dt \ge \frac 1 2 (y-x) >0.
$$
Taking the limit $n\to\infty$ implies that $H$ is strictly increasing on $[0,a]$. Being continuous, $H$ is invertible and \eqref{conj}
holds by Corollary \ref{cor}.
\end{proof}

\begin{remark}
Under additional assumption that $f$ is strictly increasing on the whole $\Real_+$, the function $H$ is furthermore a global conjugacy,
i.e. \eqref{conj} holds on $\Real_+$.
\end{remark}

\medskip

The next lemma establishes differentiability of $H$.

\medskip

\begin{lemma}\label{lem:2}
$H$ has continuous derivative 
\begin{equation}\label{limlim}
H'(x)= \prod_{j=1}^{\infty} \frac 1 \rho f'(H(x \rho^{-j})),\quad \forall x\in \Real_+
\end{equation}
where the series converges uniformly on compacts.
\end{lemma}

\begin{proof}
\

\medskip

\noindent 
\underline{Step 1}.
Let us first argue that the infinite product in \eqref{limlim} 
\begin{equation}\label{G}
G(x) := \prod_{j=1}^{\infty} \frac 1 \rho f'(H(x \rho^{-j}))
\end{equation}
is well defined. By assumption \eqref{a3}, $f$ is $\rho$-Lipschitz and hence $f^n$ is $\rho^n$-Lipschitz.
Consequently, $H_n$ is $1$-Lipschitz for all $n\in \mathbb N$ and so is $H$. This will be used in the proof on several
occasions.
Let $c:=\|f''\|_\infty$ and $r:=\frac 1 2\rho/c$, then 
\begin{equation}\label{ftaglb}
f'(x) \ge \rho -cx>0, \quad \forall x\in [0,r].
\end{equation}
For $x>0$ define the function $j(x):=[\log_\rho (x/r)]$. Then for any $j>j(x)$,   
\begin{equation}\label{estim}
\begin{aligned}
\Big|\log  \frac 1 \rho f'(H(x \rho ^{-j}))\Big| = & -\log  \frac 1 \rho f'(H(x \rho ^{-j})) \le 
-\log \Big(1-\frac c \rho  H(x \rho ^{-j})\Big) \stackrel\dagger\le \\
&
 2 \frac c \rho H(x \rho ^{-j}) \le 2 \frac c \rho  x \rho ^{-j}  =: C x\rho^{-j},
\end{aligned}
\end{equation}
where $\dagger$ holds since $-\log (1-u)\le 2u$ for all $u\in [0,\frac 1 2]$. 
The partial products in \eqref{G} can be written as 
\begin{align*}
&
G_n(x) :=  \prod_{j=1}^{n} \frac 1 \rho f'(H(x \rho^{-j})) =\\
&
\left(
\prod_{j=1}^{j(x)} \frac 1 \rho f'(H(x \rho^{-j})) 
\right)
\exp \left(\sum_{j=j(x)+1}^n \log \frac 1 \rho f'(H(x \rho^{-j}))\right) =: T(x)\exp(L_n(x)).
\end{align*}
In view of the estimate \eqref{estim},  $G_n(x)$ converges to $G(x):=T(x)\exp(L(x))$ for any $x\in \Real_+$ where $L(x)=\lim_n L_n(x)$. 
Furthermore,
\begin{equation}\label{LTG}
\begin{aligned}
\big|G_n(x)-G(x)\big| = &  |T(x)| \big|\exp(L_n(x))-\exp(L(x))\big| \le \\
& \exp(L(x)\vee L_n(x))\big| L(x) -L_n(x) \big|
\end{aligned}
\end{equation}
where we used the bound $|T(x)|\le 1$.
For any $R>0$ and all $x\in [0,R]$ the estimate \eqref{estim} implies
$$
\big| L(x) -L_n(x) \big| =
\sum_{j=n+1}^\infty \big|\log \frac 1 \rho f'(H(x \rho^{-j}))\big| \le \sum_{j=n+1}^\infty C x\rho^{-j} \le C R \frac{\rho^{-n}}{\rho-1},
$$
and thus, in view of the bound \eqref{LTG}, we obtain
\begin{equation}\label{GGn}
\sup_{x\le R} \big|G_n(x)-G(x)\big|\to 0.
\end{equation}
Since $G_n$ is continuous for any $n$, this uniform convergence implies that $G$ is continuous as well.

\medskip
\noindent
\underline{Step 2.}
To show that $H(x)$ is differentiable and to verify the claimed formula for the derivative,
it remains to show that the sequence of derivatives
$$
H_n'(x)  = \prod_{j=1}^n \frac 1 \rho f'(f^{n-j}(x/\rho^n))
$$
converges to $G$ uniformly on compacts. Fix an $R>0$, define $J(R)=[\log_\rho (R/r)]$ and, for $n>J(R)$, let
$$
\widetilde P_n(x) := \prod_{j=1}^{J(R)} \frac 1 \rho f'(f^{n-j}(x/\rho^n)),\quad
P_n(x):= \prod_{j=J(R)+1}^n \frac 1 \rho f'(f^{n-j}(x/\rho^n))
$$
and 
$$
\widetilde Q_n(x):= \prod_{j=1}^{J(R)} \frac 1 \rho f'(H(x \rho^{-j})),
\quad
Q_n(x):= \prod_{j=J(R)+1}^n \frac 1 \rho f'(H(x \rho^{-j})).
$$
Since $\|f'\|_\infty=\rho$ all these functions are bounded by 1 and 
\begin{align*}
&
\big|H_n'(x) - G(x)\big|\le 
\big|H_n'(x) - G_n(x)\big|
+
\big|G_n(x) - G(x)\big| = \\
&
\big|\widetilde P_n(x)  P_n(x) - \widetilde Q_n(x) Q_n(x)\big|
+
\big|G_n(x) - G(x)\big| \le \\
&
\big|  P_n(x)-  Q_n(x) \big| 
+
\big|\widetilde P_n(x)   - \widetilde Q_n(x)\big| 
+
\big|G_n(x) - G(x)\big|. 
\end{align*}
Since $f'$ is continuous and the convergence $H_n\to H$ is uniform on compacts, it follows that 
$$
\sup_{x\le R}\big|\widetilde P_n(x)   - \widetilde Q_n(x)\big|
= \sup_{x\le R}\left|
\prod_{j=1}^{J(R)} \frac 1 \rho f'(H_{n-j}(x\rho^{-j}))
-
\prod_{j=1}^{J(R)} \frac 1 \rho f'(H(x \rho^{-j}))
\right| 
\xrightarrow[n\to\infty]{}0,
$$
and hence, to complete the proof, we need to show that 
\begin{equation}
\label{need2}
\sup_{x\le R}\big|  P_n(x)   -   Q_n(x)\big|\xrightarrow[n\to\infty]{}0.
\end{equation} 
To this end, in view of Corollary \ref{cor},
\begin{align*}
H(x\rho^{-j}) =\, & f\circ H(x\rho^{-j-1})) = f^2\circ H(x\rho^{-j-2})) = ... =\\
& f^{n-j}\circ H(x\rho^{-j-(n-j)})= 
f^{n-j}\circ H(x\rho^{ - n })
\end{align*}
and hence    
$$
P_n(x) - Q_n(x) = 
\prod_{j=J(R)+1}^n \frac 1 \rho f'(f^{n-j}(x \rho^{-n})) - \prod_{j=J(R)+1}^n \frac 1 \rho f'\big(f^{n-j}\big(H(x \rho^{-n})\big)\big).
$$
Consequently, for all $x\in (0,R]$, 
\begin{equation}\label{longeq}
\begin{aligned}
&
\big|\log P_n(x) - \log Q_n(x)\big| \le \\
&
\sum_{j=J(R)+1}^n \Big| \log \frac 1 \rho f'(f^{n-j}(x\rho^{-n})) -   \log\frac 1 \rho f'\big(f^{n-j}\big(H(x\rho^{ - n })\big)\big)\Big| \stackrel\dagger\le \\
&
\frac 1 {\rho - cr}  \|f''\|_\infty\sum_{j=J(R)+1}^n 
\big|    f^{n-j}(x\rho^{-n})  -    f^{n-j}\big(H(x\rho^{ - n })\big) \big| \le \\
&
\frac 1 {\rho - cr}  \|f''\|_\infty
\sum_{j=1}^n \rho^{ n-j}\big|     x\rho^{-n}    -       H(x\rho^{ - n })  \big|  \le 
C  \rho^n \big|   \rho^{-n}  x    -      H(x\rho^{ - n })  \big| = \\
&
C \rho^n \big|  H\circ  H^{-1}( x\rho^{-n} )    -      H(x\rho^{ - n })  \big| \stackrel \ddagger\le 
C  \big|  \rho^n H^{-1}( x\rho^{-n} )    -       x    \big|.
\end{aligned}
\end{equation}
Here the bound $\dagger$ holds since for $j> J(R)$ both arguments of $f'$ are smaller than $r$ 
and thus \eqref{ftaglb} applies. 
The inequality $\ddagger$ is true since $H$ is $1$-Lipschitz. 
The inverses in the last line of \eqref{longeq} are well defined for $n\ge k:=[\log_\rho (R/H(a))]+1$
where $a$ is the constant guaranteed by Lemma \ref{lem:8}. Moreover, for all such $n$
\begin{align}\label{conv}
& \big|  \rho^n H^{-1}(  x\rho^{-n})    -       x    \big| =
\rho^k \big|  \rho^{n-k} H^{-1}(  x\rho^{-k}\rho^{-(n-k)})    -       x\rho^{-k}    \big| =\\
&\nonumber
\rho^k \big|   H^{-1}\circ f^{n-k}( x\rho^{-k}\rho^{-(n-k)} )    -       x \rho^{-k}    \big| = 
\rho^k \big|   H^{-1}\circ H_{n-k}(x\rho^{-k})    -       x \rho^{-k}   \big| \xrightarrow[n\to\infty]{}0.
\end{align}
Moreover, the sequence of functions $D_n(x) := \rho^n H^{-1}(  x\rho^{-n})$ is decreasing on $[0,R]$ for all $n$ large enough:
$$
D_{n+1}(x) = \rho^n \rho H^{-1}(  x\rho^{-n-1}) = \rho^n   H^{-1}\circ f (  x\rho^{-n-1})\le 
\rho^n   H^{-1} (  x\rho^{-n }) = D_n(x),
$$
where the inequality holds since $H^{-1}$ is increasing near the origin. 
It follows now from Dini's theorem that the convergence in \eqref{conv} is uniform:
$$
\sup_{x\le R} \big|  \rho^n H^{-1}( x\rho^{-n} )    -       x    \big|\xrightarrow[n\to\infty]{}0.
$$
The convergence in \eqref{need2} holds since both $Q_n$ and $P_n$ are bounded by 1 and
$$
\sup_{x\le R}\big|  P_n(x)   -   Q_n(x)\big| \le \sup_{x\le R}\big|P_n(x) \vee  Q_n(x)\big| \sup_{x\le R}\big|  \log P_n(x)   -  \log Q_n(x)\big|
\xrightarrow[n\to\infty]{}0.
$$
\end{proof}

\begin{lemma}
$H$ has continuous second derivative 
\begin{equation}\label{H2d}
H''(x) = 
H'(x) \sum_{i=1}^\infty \frac {f''(H(x \rho^{-i}))}{ f'(H(x \rho^{-i}))}  H'(x \rho^{-i})\rho^{-i}.
\end{equation}
\end{lemma}

\begin{proof}
The partial products in \eqref{limlim}  
$$
G_n(x): =\prod_{j=1}^n \frac 1 \rho f'(H(x \rho^{-j}))
$$
satisfy 
\begin{align*}
G'_n(x) = & \sum_{i=1}^n\left(\prod_{j=1, j\ne i}^n \frac 1 \rho f'(H(x \rho ^{-j}))\right)
\frac 1 \rho f''(H(x \rho^{-i}))H'(x \rho^{-i}) \rho ^{-i}= \\
&
G_n(x) \sum_{i=1}^n \frac {f''(H(x \rho^{-i}))}{ f'(H(x \rho^{-i}))}  H'(x \rho^{-i}) \rho^{-i},
\end{align*}
where the convention $0/0=0$ is used.
By assumption \eqref{a3}, $f''/f'$ is bounded uniformly on a vicinity of the origin. 
$H'$ is continuous  by Lemma \ref{lem:2} and therefore is bounded on compacts. Hence the series is compactly convergent. 
By Lemma \ref{lem:2}, so is $G_n$. Thus $G'_n(x)$ converges compactly, its limit is continuous and coincides with $H''(x)$.  
\end{proof}

\subsection{The auxiliary Galton-Watson process}\label{sec:3.2}
 
Let $(U_{n,j}: n\in \mathbb N, j\in \mathbb Z_+)$ be an array of i.i.d. $U([0,1])$ random variables and 
define   
$$
\xi_{n,j}(x) = F^{-1}_x(U_{n,j}) := \min\big\{t\ge 0: F_x(t)\ge U_{n,j}\big\},
$$
where $F_x(t)$ is the offspring distribution function when the population density is $x$, cf. assumption \eqref{a1}. 
Then $\P(\xi_{n,j}(x)=k)=p_k(x)$ for all $k\in \mathbb Z_+$.
Let $\eta_{n,j}:=\xi_{n,j}(0)$.
By assumption \eqref{a1} 
\begin{equation}\label{xieta}
\xi_{n,j}(x) \le \eta_{n,j}\quad \forall x\in \Real_+,\ n,j\in \mathbb N.
\end{equation}
Let $Z=(Z_n, n\in \mathbb Z_+)$ and $Y=(Y_n, n\in \mathbb Z_+)$ be processes generated by the recursions 
$$
Z_n =    \sum_{j=1}^{Z_{n-1}}\xi_{n,j}(\overline Z_{n-1})\quad \text{and}\quad
Y_n =   \sum_{j=1}^{Y_{n-1}} \eta_{n,j}
$$
started from the same initial conditions $Z_0=Y_0=1$. By construction these processes coincide in distribution with \eqref{Z} and 
\eqref{Y} respectively. Moreover, in view of \eqref{xieta}, by induction 
\begin{equation}\label{ZYineq}
Z_n \le Y_n, \quad \forall n\in \mathbb Z_+.
\end{equation}

\subsection{The approximation}\label{sec-T}

In view of  \eqref{H},
$$
 \overline Y_{n_1} - W\rho^{-\{\log_\rho K\}}  =   \rho^{-\{\log_\rho K\}}\big(\rho^{-n_1} Y_{n_1}- W\big)
 =O_\P(\rho^{-n_1/2}) = O_\P(K^{-1/2}).
$$
Since $H$ has continuous derivative it follows that 
$$
H( \overline Y_{n_1})- H(W\rho^{-\{\log_\rho K\}}) = O_\P(K^{-1/2}).
$$
Thus to prove the assertion of Theorem \ref{thm:main} it remains to show that 
$$
 \overline Z_{n_1}  - H(\overline Y_{n_1}) = O_\P(K^{-1/2}\log K), \quad K\to \infty.
$$
The process $\overline Y_n = K^{-1} Y_n$ satisfies 
$$
\overline Y_n = \rho \overline Y_{n-1} + \frac 1 K \sum_{j=1}^{Y_{n-1}}(\eta_{n,j}-\rho).
$$
By Taylor's approximation and in view of Corollary \ref{cor}
\begin{equation}\label{HYn}
\begin{aligned}
H(\overline Y_n) =\, & H(\rho \overline Y_{n-1}) + H'(\rho \overline Y_{n-1})\frac 1 K \sum_{j=1}^{Y_{n-1}}(\eta_{n,j}-\rho)
+ R_n(K) =\\
&
f(H( \overline Y_{n-1})) + H'(\rho \overline Y_{n-1})\frac 1 K \sum_{j=1}^{Y_{n-1}}(\eta_{n,j}-\rho)
+ R_n(K)
\end{aligned}
\end{equation}
where 
\begin{equation}\label{RnK}
R_n(K) := \frac 1 2 H''(\theta_{n-1}(K))\left(\frac 1 K \sum_{j=1}^{Y_{n-1}}(\eta_{n,j}-\rho)\right)^2
\end{equation}
with $\theta_{n-1}(K)\ge 0$ satisfying 
\begin{equation}\label{theta}
\big|\theta_{n-1}(K) - \rho \overline Y_{n-1}\big|\le \left|\frac 1 K \sum_{j=1}^{Y_{n-1}}(\eta_{n,j}-\rho)\right|.
\end{equation}
Since $\|f'\|_\infty= \rho$ is assumed, $f$ is $\rho$-Lipschitz. By subtracting equation  \eqref{Zbar} from \eqref{HYn} we obtain 
the bound for $\delta_n := | H(\overline Y_n)-\overline Z_n|$:
\begin{equation}\label{delta}
\delta_n   \le \rho \delta_{n-1}
+ \big|\eps^{(1)}_n\big|
+ \big|\eps^{(2)}_n\big|
+\big|\eps^{(3)}_n\big|
+ |R_n(K)|
\end{equation}
subject to $\delta_0 = |H(1/K)-1/K|$, where we defined
\begin{align*}
\eps^{(1)}_n & = \Big(H'(\rho \overline Y_{n-1})-1\Big)\frac 1 K \sum_{j=1}^{Y_{n-1}}(\eta_{n,j}-\rho), \\
\eps^{(2)}_n & = \frac 1 K \sum_{j=1}^{Z_{n-1}}\Big((\eta_{n,j}-\rho) -   (\xi_{n,j}(\overline Z_{n-1})-m(\overline Z_{n-1}))\Big),\\
\eps^{(3)}_n & = \frac 1 K \sum_{j=Z_{n-1}+1}^{Y_{n-1}}(\eta_{n,j}-\rho).
\end{align*}
Consequently,  
$$
\delta_n \le \rho^n \delta_0 + \sum_{j=1}^n \rho^{n-j} \Big(\big|\eps^{(1)}_j\big|
+ \big|\eps^{(2)}_j\big|
+\big|\eps^{(3)}_j\big|
+ |R_j(K)|
\Big) 
$$
and it is left to show that the contribution of each term at time $n_1=[\log_ \rho K]$ is of order $O_\P(K^{-1/2}\log K)$
as $K\to\infty$.

\medskip

\subsubsection{Contribution of the initial condition}
Since $H(0)=0$ and, by \eqref{limlim}, $H'(0)=1$, Taylor's approximation implies that for all $K$ large enough
$$
\delta_0  = |H(1/K)-1/K| \le \frac 1 2\sup_{x\le 1}|H''(x)| K^{-2} = CK^{-2}
$$
and, consequently, 
$
|\rho^{n_1} \delta_0| \le CK^{-1}.
$
\medskip

\subsubsection{Contribution of $R_n(K)$}\label{subs:R}
To estimate the residual, defined in \eqref{RnK}, let us show first that the family of random variables
\begin{equation}\label{htt}
\max_{m\le n_1} \Big|H''(\theta_m(K))\Big|
\end{equation}
is bounded in probability as $K\to\infty$. By equation \eqref{theta},
\begin{align*}
&
\E  \theta_{n-1}(K)  \le\  
\rho\E  \overline Y_{n-1}  + \E\left| \frac 1 K \sum_{j=1}^{Y_{n-1}} \big(\eta_{n,j}-\rho\big)\right|\le \\
&
  \rho \E \overline Y_{n-1} + \frac 1 K \sqrt{\E Y_{n-1} \sigma^2(0)}\le 
\frac 1 K\rho^n + \frac 1 K \sqrt{ \rho^{ n} \sigma^2(0)} \le \\
&   K^{-1} \rho^n + C K^{-1}\rho^{n/2} \le  2CK^{-1} \rho^n.
\end{align*}
If $H''$ is bounded then \eqref{htt} is obviously bounded. 
Let us proceed assuming that $H''$ is unbounded. Define 
$
\psi(M):= \max_{x\le M} |H''(x)|.
$ 
By continuity, $\psi(M)$ is finite, continuous and increases to $\infty$. 
Let  $\psi^{-1}$ be its generalized inverse 
$$
\psi^{-1}(t)=\inf\{x\ge 0: \psi(x)\ge t\}.
$$
Since $\psi$ is continuous and unbounded, $\psi^{-1}$ is nondecreasing (not necessarily continuous) and 
$\psi^{-1}(t)\to \infty$ as $t\to\infty$.
Then for any $R\ge 0$, by the union bound,
\begin{align*}
&
\P \Big(\max_{m\le n_1} |H''(\theta_m(K)|\ge R\Big)\le 
\P \Big(\max_{m\le n_1} \psi(\theta_m(K))\ge R\Big) \le \\
&
\sum_{m=1}^{n_1} \P \Big( \psi(\theta_m(K))\ge R\Big)
\le 
\sum_{m=1}^{n_1} \P \Big(   \theta_m(K) \ge \psi^{-1}(R)\Big) \le \\
&
\sum_{m=1}^{n_1} \frac{\E\theta_m(K)} { \psi^{-1}(R)} \le 
\frac {1} { \psi^{-1}(R)} \sum_{m=1}^{n_1}  2C K^{-1} \rho^m   \le  \frac {\rho}{\rho-1}\frac {2C} { \psi^{-1}(R)}\xrightarrow[R\to\infty]{}0.
\end{align*}
This proves that \eqref{htt} is bounded in probability. 
The contribution of $R_n(K)$ in \eqref{delta} can now be bounded as  
\begin{align*}
\left|\sum_{m=1}^{n_1} \rho^{n_1-m} R_m(K)\right| \le 
\max_{j\le n_1}\Big| H''(\theta_j(K))\Big|
 \sum_{m=1}^{n_1} \rho^{n_1-m} 
\left(\frac 1 K \sum_{j=1}^{Y_{m-1}}(\eta_{m,j}-\rho)\right)^2
\end{align*}
where 
\begin{align*}
&
\E  \sum_{m=1}^{n_1} \rho^{n_1-m} 
\left(\frac 1 K \sum_{j=1}^{Y_{m-1}}(\eta_{m,j}-\rho)\right)^2
=\\
&
\sum_{m=1}^{n_1} \rho^{n_1-m}  \frac 1 {K^2} \E Y_{m-1} \sigma^2(0)  \le 
\sum_{m=1}^{n_1} \rho^{n_1-m}  \frac 1 {K^2} \rho^{m} \sigma^2(0) \le 
C K^{-1}\log K.
\end{align*}
Hence 
$$
\left|\sum_{m=1}^{n_1} \rho^{n_1-m} R_m(K)\right| = O_\P(1) O_\P(K^{-1} \log K) = O_\P(K^{-1} \log K).
$$

\subsubsection{Contribution of $\eps^{(3)}$}
By conditional independence of $\eta_{n,j}$'s 
$$
\E\big(\eps^{(3)}_m\big)^2 =  \frac {\sigma^2(0)} {K } \E (\overline Y_{m-1}-\overline Z_{m-1}) .
$$
In view of \eqref{ZYineq}, the sequence $D_m := \overline Y_{m}-\overline  Z_{m}\ge 0$ satisfies 
\begin{align*}
\E D_m =\, & \frac 1 K\E \left(\sum_{j=1}^{Y_{m-1}}\eta_{m,j}-\sum_{j=1}^{Z_{m-1}}\xi_{m,j}(\overline Z_{m-1})\right) =\\
&
\frac 1 K\E  \sum_{j=Z_{m-1}+1}^{Y_{m-1}}\eta_{m,j}+\frac 1 K\E \sum_{j=1}^{Z_{m-1}}\big(\eta_{m,j}-\xi_{m,j}(\overline Z_{m-1})\big) =\\
&
\rho \E D_{m-1} + \frac 1 K\E\sum_{j=1}^{Z_{m-1}} \big(\rho-m(\overline Z_{m-1})\big) =\\
&
\rho \E D_{m-1} +   \E \overline Z_{m-1} \big(\rho-m(\overline Z_{m-1}) \big) =  \\
&
\rho \E D_{m-1} +  \E \big(\rho\overline Z_{m-1}-f(\overline Z_{m-1})\big) \le \\
&
\rho \E D_{m-1} +  \frac 1 2 \|f''\|_\infty \E \overline Z_{m-1}^2 \le \\
&
\rho \E D_{m-1} + C K^{-2}  \rho^{2m},
\end{align*} 
where the last bound holds in view of \eqref{ZYineq} and the well known formula 
for the second moment of the Galton-Watson process.
Since $D_0=0$ it follows that 
$$
\E D_m \le  \sum_{\ell=1}^m \rho^{m-\ell} C K^{-2}  \rho^{2\ell} \le C K^{-2}\rho^{2m}.
$$
Hence the contribution of $\eps^{(3)}$ in \eqref{delta} is bounded by   
\begin{align*}
\E \Big|\sum_{m=1}^{n_1} \rho^{n_1-m} \eps^{(3)}_m\Big| \le\, & C \sum_{m=1}^{n_1} \rho^{n_1-m} K^{-1/2} \sqrt{\E D_m}  \le \\
&
C \sum_{m=1}^{n_1} \rho^{n_1-m} K^{-1/2} K^{-1}\rho^{ m}
\le C K^{-1/2} \log K.
\end{align*}

\subsubsection{Contribution of $\eps^{(2)}$}
By assumption \eqref{a2},
\begin{align*}
\E \big(\eps^{(2)}_m\big)^2 =\, &
 K^{-2}\E  \sum_{j=1}^{Z_{m-1}} \big( \eta_{m,j} -\xi_{m,j}(\overline Z_{m-1}) - \big(\rho   -m(\overline Z_{m-1})\big) \big)^2 \le \\
&
 K^{-2}\E  \sum_{j=1}^{Z_{m-1}} \big(\eta_{m,j} -\xi_{m,j}(\overline Z_{m-1})\big)^2  \stackrel\dagger{\le}
 K^{-2}\E  \sum_{j=1}^{Z_{m-1}} (m_2(0)  - m_2(\overline Z_{m-1})\big) \le \\
 &
K^{-2}\E  \sum_{j=1}^{Z_{m-1}} L \overline Z_{m-1}  \le C K^{-3}\rho^{2m}
\end{align*}
where $\dagger$ holds by \eqref{xieta}.
Hence $\eps^{(2)}$ contributes 
$$
\E \Big|\sum_{m=1}^{n_1} \rho^{n_1-m} \eps^{(2)}_m\Big| \le  
C \sum_{m=1}^{n_1} \rho^{n_1-m}  K^{-3/2}\rho^{ m} \le C K^{-1/2}\log K.
$$

\subsubsection{Contribution of $\eps^{(1)}$}
The function
$
g(x): = H'(x)-1
$
is continuously differentiable  with $g(0)=0$ and thus Taylor's approximation gives
$$
\eps^{(1)}_n =\,  
g(\rho \overline Y_{n-1})\frac 1 K \sum_{j=1}^{Y_{n-1}}(\eta_{n,j}-\rho) =
g'(\zeta_{n-1}(K)) \rho \overline Y_{n-1}\frac 1 K \sum_{j=1}^{Y_{n-1}}(\eta_{n,j}-\rho)  
$$
where 
$
\zeta_{n-1}(K)
$
satisfies $0\le \zeta_{n-1}(K)\le \rho \overline Y_{n-1}$.
Here 
\begin{align*}
&
\E\Big|\overline Y_{n-1}\frac 1 K \sum_{j=1}^{Y_{n-1}}(\eta_{n,j}-\rho)\Big|  \le \\
&
\left(\E \big(\overline Y_{n-1}\big)^2\right)^{1/2} 
\left(
\E \Big(\frac 1 K \sum_{j=1}^{Y_{n-1}}(\eta_{n,j}-\rho)\Big)^2 
\right)^{1/2}\le \\
&
\big(K^{-2}\rho^{2n})^{1/2}
\Big(
K^{-2} \E Y_{n-1} \sigma^2(0)
\Big)^{1/2} \le  C K^{-2} \rho^{3/2n}.
\end{align*}
It follows that  
$$
\E \sum_{m=1}^{n_1} \rho^{n_1-m} \Big|\overline Y_{m-1}\frac 1 K \sum_{j=1}^{Y_{m-1}}(\eta_{m,j}-\rho)\Big| \le  \sum_{m=1}^{n_1} \rho^{n_1-m} C K^{-2} \rho^{3/2m} \le C K^{-1/2}.
$$
It is then argued as in Subsection \ref{subs:R} that  
$$
 \sum_{m=1}^{n_1} \rho^{n_1-m} \eps^{(1)}_m = O_\P(1) O_\P(K^{-1/2}) =O_\P(K^{-1/2}).
$$

\subsection*{Acknowledgements}

The research was supported by ARC grant DP220100973.


\end{document}